\documentclass[journal,twoside,web]{ieeecolor}

\usepackage{generic}
\usepackage{cite}
\usepackage{amsmath,amssymb,amsfonts}
\usepackage{algorithmic}
\usepackage{graphicx}
\usepackage{textcomp}
\usepackage{subfigure}
\usepackage{epsfig} 
\usepackage{cite}
\usepackage{lcsys}

\usepackage[utf8]{inputenc}
\usepackage[T1]{fontenc}
\usepackage{amsthm}
\usepackage{mathrsfs}
\usepackage{mathtools}
\usepackage[short]{optidef}
\usepackage{textcmds} % for quotation marks
\usepackage{xcolor}
\usepackage{xspace}

\usepackage{tikz}
\tikzstyle{round} = [circle, minimum size=0.25cm, text centered, draw=black, thick]
\tikzstyle{arrow} = [thick,->,>=stealth]
\tikzset{weightnode/.style n args={3}{%
  circle,
  draw,
  inner sep=2.5mm,
  append after command={%
  \pgfextra{ %
   \draw (\tikzlastnode.center) -- (\tikzlastnode.south) ;
   \draw (\tikzlastnode.west)   -- (\tikzlastnode.east) ;
   \path (\tikzlastnode.center) -- node[black] {\footnotesize #1} (\tikzlastnode.north);  % node label
   \path (\tikzlastnode.center) -- node[myblue] {\footnotesize #2} (\tikzlastnode.south west); % cost
   \path (\tikzlastnode.center) -- node[myorange] {\footnotesize #3} (\tikzlastnode.south east); } }} % bound
}

\newtheorem{theorem}{Theorem}
\newtheorem{lemma}{Lemma}
\newtheorem{proposition}{Proposition}

\theoremstyle{definition}
\newtheorem{assumption}{Assumption}

\theoremstyle{remark}
\newtheorem{remark}{Remark}
\newtheorem{example}{Example}

\definecolor{myblue}{HTML}{1F77B4}
\definecolor{myorange}{HTML}{FF7F0E}
\definecolor{mygreen}{HTML}{2CA02C}
\definecolor{myred}{HTML}{D62728}
\definecolor{mypurple}{HTML}{9467bd}

\newcommand*{\ie}{i.e.,\@\xspace}
\newcommand*{\eg}{e.g.,\@\xspace}
\newcommand*{\cf}{cf.\@\xspace}

\newcommand{\R}{\mathbb{R}}
\newcommand{\N}{\mathbb{N}}
\newcommand*{\Matrix}[1]{\ensuremath{ \begin{pmatrix} #1 \end{pmatrix} }} 	% matrix

\begin{document}

\title{Model Predictive Control for Constrained Linear Positive Systems on Graphs}

\author{
Roland Schurig$^{a}$,
David Ohlin$^{b}$,
Anders Rantzer$^{b}$,
Emma Tegling$^{b}$,
Rolf Findeisen$^{a}$
\thanks{The authors acknowledge support by the LOEWE initiative emergenCITY and by the Wallenberg AI, Autonomous Systems and Software Program (WASP) funded by the Knut and Alice Wallenberg Foundation.}
\thanks{$^{a}$ TU Darmstadt, Control and Cyber-Physical Systems Laboratory, \{roland.schurig,rolf.findeisen\}@iat.tu-darmstadt.de}
\thanks{$^{b}$  Department of Automatic Control, Lund University. The authors are with the ELLIIT Strategic Research Area at Lund University. \{david.ohlin,anders.rantzer,emma.tegling\}@control.lth.se}
}

\maketitle
\thispagestyle{empty}

%%%%%%%%%%%%%%%%%%%%%%%%%%%%%%%%%%%%%%%%%%%%%%%%%%%%%%%%%%%%%%%%%%%%%%%%%%%%%%%%
\begin{abstract}
Positive systems describing networks with inherently non-negative states and inputs arise naturally in routing, logistics, and compartmental modelling. We consider problems modelled as positive linear systems {in incidence form} with linear cost. {The addition of} capacity constraints on states (storage) and inputs (flows between nodes) significantly increases the problem complexity. 
Leveraging the analytic structure of the unconstrained problem, an explicit suboptimal admissible controller is constructed. {This yields} graph-computable performance bounds and a minimum stabilising horizon length {for a model predictive controller without terminal conditions}. A convex program enables efficient computation of the optimal bound and horizon. These results highlight how system structure enables explicit MPC guarantees that are typically not available.
\end{abstract}


\section{Introduction}
\label{sec:introduction}

 \noindent
 \IEEEPARstart{W}{e} consider optimal control of {a subclass of} positive systems on graphs under capacity constraints on vertex storage (states) and edge flows (inputs), precluding known explicit solutions. 
 {This motivates a model predictive control (MPC)} approach without terminal ingredients that exploits the {positivity and special structure of the dynamics} to derive explicit, computable horizon bounds for closed-loop stability.

{Infinite-horizon} optimal control of positive systems has recently received increased attention. 
Recent work has established closed-form solutions, \cite{Blanchini2023},~\cite{Rantzer2022}, with extensions to optimal control of networks in~\cite{Ohlin2024},~\cite{Ohlin2025}. These works do not consider capacity constraints on states and inputs, which arise naturally in practical routing problems. Addressing these constraints is the core contribution of this work, as they render previous closed-form solutions inapplicable.

The key enabling property of the considered system class is the existence of a parametrised family of stabilising controllers derived from the optimal static policy in~\cite{Ohlin2024}, which can be made admissible under capacity constraints. This structure allows us to construct a suboptimal but feasible feedback law that guarantees constraint satisfaction.

Building on this construction, we derive explicit lower bounds on the prediction horizon required for stability of the MPC closed loop, as well as closed-form performance bounds, in the spirit of~\cite{Gruene2008}. Moreover, the sparse structure of the resulting controller class enables the computation of these bounds via a convex geometric program~\cite{Boyd2004}. Such explicit and computationally tractable horizon bounds are generally not available for constrained linear systems, and arise here as a direct consequence of the positivity and graph structure.

A large literature stemming from the foundational work of Kantorovich~\cite{Kantorovich1960} has explored resource allocation %and inventory management 
through {the lens of routing problems on graphs}. Such formulations typically consider static flow, modelled as supply and demand. In contrast, the scenario considered here is %concerned with 
the transient dynamics; locating all commodity in a single goal vertex, starting from some initial distribution. When no capacity constraints are present, this formulation falls into the class of shortest path problems. Starting from the initial work of Dijkstra~\cite{Dijkstra1959}, numerous algorithms have been devised to treat this problem. The typical approach (see \eg~\cite{Barto1995}) is to model the state as a probability distribution over the vertex set, a setting that does not naturally lend itself to fixed capacity constraints.  

{The previous work~\cite{Moss1982}, seeking to treat graph routing using the theory of optimal control, defines a linear program routing all commodities to the goal in finite time. 
In the present setting, however, it is unclear how long this horizon needs to be to ensure feasibility, or what level of performance can be guaranteed compared to the infinite-horizon case.}
For this reason, we use online optimisation {in an MPC scheme, which is well established for constrained systems \cite{Gruene2017,Borrelli2017}.%\cite{Gruene2017,Borrelli2017,lucia2016predictive}.
} 

Standard MPC methods use terminal conditions to ensure stability of the closed loop; see \eg \cite{Gruene2017, Borrelli2017,Chen1998}. 
This yields a simplified analysis but more conservative suboptimality estimates. A differing approach (see \eg~\cite[Chapter~6]{Gruene2017}) removes the terminal conditions and instead bounds the optimal value function on the state space, giving global performance guarantees. 
This is the approach adopted. 
{The main contribution of this work is a structured approach to designing an MPC scheme without terminal conditions, but with guaranteed stability and infinite-horizon performance bounds for a specific class of constrained positive systems on graphs.}
{These bounds are explicit and computable} via a convex optimisation method, {which is generally 
not possible} for broader system classes~\cite{Boccia2014}. 
{We exploit the incidence structure of the system dynamics, the box-type capacity constraints and the optimal routing structure for the unconstrained case, as inherited from \cite{Ohlin2024}.} 

{
    A routing problem similar to our convex program was considered in \cite{Arneson2009}. 
    The authors propose a sufficient condition for constraint satisfaction, which is further relaxed to obtain a linear program. 
    In contrast, we work with necessary and sufficient conditions, which lead to a geometric program.
}

Several works have considered MPC approaches for positive systems and systems on graphs.  
In~\cite{Braun2003}, an MPC framework is developed for general dynamics and costs under network constraints. 
For uncertain positive systems,~\cite{Zhang2018} proposes a model predictive approach with linear performance index, which is the objective function studied in this work. 
The recent work~\cite{Lejarza2021} uses tube-based MPC to control a class of noisy linear system with graph constraints, also using linear cost. 
The previous work that most closely resembles the present is~\cite{Mehrivash2019}. 
It takes a similar approach to the one presented here, but for a different class of positive systems. 
Our choice of system class is motivated by the simple analytic expressions available for the optimal cost and controller in the unconstrained case. It is precisely this that allows for the derivation of concrete bounds on horizon length and performance, which are not present in the mentioned prior studies. 

\textit{Notation:} To denote non-negative real numbers, we use $\R_{\geq 0} \coloneqq \{ x \in \R \mid x \geq 0 \}$, and analogously $\R_{> 0}$ for the positive real numbers. 
The positive integers are $\mathbb{N}$, and ${\mathbb{N}_0 \coloneqq \mathbb{N} \cup \{ 0 \}}$. 
The $i$-th canonical basis vector in $\R^n$ is $e_i$. %, where the dimension $n$ will be clear from context.
Inequalities are applied element-wise for elements of $\R^n$, and the vector of all ones in $\R^n$ is $\mathbf{1}$. 

\section{Problem setup}

\noindent
We consider a special class of linear positive systems: 
\begin{equation}\label{eq:dynamics}
    x(t+1) = x(t) + Bu(t)
\end{equation}
with state space ${\mathbb{R}_{\geq 0}^n}$ and control space ${\R_{\geq 0}^{m}}$. This can be interpreted as dynamics over a directed graph: 
Take the vertex set ${\mathcal{V} = \{1, \dots, n\}}$, each representing a buffer that can store a commodity. 
The goal is to transport all commodities present in the network to a goal vertex $g$, which we denote ${g = n+1}$.
Commodities are routed through the network via the edges encoded in the incidence matrix 
\begin{equation} \label{eq:B}
	B = \Matrix{B_1 & \dots & B_n},
\end{equation}
with $B_i \in \R^{n \times m_i}$, so that ${\sum_i m_i = m}$. 
Each column of the local incidence matrix $B_i$ has an element $-1$ on the $i$-th row and 1 on the row corresponding to a \textit{successor} vertex $j$, if $j \in \mathcal{V}$. 
This edge transports commodity from $i$ to $j$. 
Any edge leading to the goal vertex $g$ is modelled as a column of~$B$ with only a single non-zero element~$-1$ (see Example~1). 
This edge removes commodity from the network.

If vertex $i \in \mathcal{V}$ has more than one outgoing edge, the columns in $B_i$ are ordered according to the label of the respective successor vertex ${j \in \mathcal{V} \cup \{n+1\}}$. %\roland{Is that formulation clear?}
Locating all commodity in the goal vertex is represented in~\eqref{eq:dynamics} by the state $x = 0$. 
Given a state ${x \in \R_{\geq 0}^n}$, the control value must be contained in the set
\begin{equation}\label{eq:U(x)}
	U(x) \coloneqq \{ (u_1, \dots, u_n) \in \R_{\geq 0}^m \mid \forall i \in \mathcal{V} \colon \mathbf{1}^\top u_i \leq x_i  \} ,
\end{equation}
where $x_i$ denotes the $i$-th element of $x$, and the partition of $u$ follows the partition of $B$. 
{This constraint ensures that the state remains nonnegative and prevents more commodities being taken out of a vertex than are currently present.}

\begin{example} \label{ex:incidence_matrix}
    Consider the graph in Fig.~\ref{fig:shortest_path_problem}. %, and ignore the coloured edges for now. 
    The goal vertex is $n+1=6$, so the corresponding incidence matrix becomes
    \begin{equation*}
        B = \left(\begin{smallmatrix}
            -1 & -1 & -1 & 0 & 0 & 0 & 0 & 0 & 0 \\
            1 & 0 & 0 & -1 & -1 & -1 & 0 & 0 & 0 \\
            0 & 0 & 0 & 1 & 0 & 0 & -1 & 1 & 1 \\
            0 & 1 & 0 & 0 & 1 & 0 & 0 & -1 & 0 \\
            0 & 0 & 1 & 0 & 0 & 1 & 0 & 0 & -1
        \end{smallmatrix}\right) . 
    \end{equation*}
\end{example}

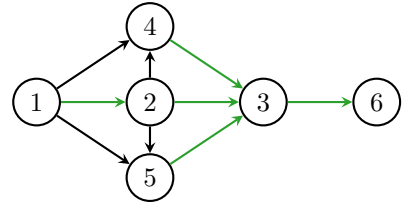
\begin{figure}[t]
    \centering
    \begin{tikzpicture}[node distance=1.5cm]
        %Nodes
        \node at (0,0) (n1) [round] {$1$};
        \node (n2) [round, right of=n1] {$2$};
        \node (n3) [round, right of=n2] {$3$};
        \node (n4) [round, above of=n2, yshift=-0.5cm] {$4$};
        \node (n5) [round, below of=n2, yshift=0.5cm] {$5$};
        \node (n6) [round, right of=n3] {$6$}; % goal node
        
        % edges
        % leaving #1
        \draw[arrow,mygreen] (n1) -- (n2); % optimaö
        \draw[arrow] (n1) -- (n4);
        \draw[arrow] (n1) -- (n5);
        % leaving #2
        \draw[arrow,mygreen] (n2) -- (n3); % optimal
        \draw[arrow] (n2) -- (n4);
        \draw[arrow] (n2) -- (n5);
        % leaving 3
        \draw[arrow,mygreen] (n3) -- (n6); % optimal
        % leaving #4
        \draw[arrow,mygreen] (n4) -- (n3); % optimal
        % leaving #5
        \draw[arrow,mygreen] (n5) -- (n3); % optimal
    \end{tikzpicture}
    \caption{
      Graph with $n=5$ vertices and goal vertex $g=6$. Coloured 
edges indicate inputs used in the optimal unconstrained 
solution. Example~1 uses costs $s^\top\!= (10 \; 5 \; 1 
\; 3 \; 2)$, $r^\top\! = (1 \; 5 \; 5 \; 1 \; 1 \; 1 \; 
1 \; 1 \; 1)$ and capacity constraints $\bar{x} = 
\mathbf{1}$, $\bar{u}^\top\! = (\frac{1}{4}\;\frac{1}{4}\;
 \frac{1}{4}\;\frac{1}{4}\;\frac{1}{4}\;\frac{1}{4}\;1\;1\;1)$.}
    \vspace*{-5mm}
    \label{fig:shortest_path_problem}
\end{figure}

\noindent {\it Optimal Control Problem:} 
This work concerns the optimal control of the presented system class. 
The nature of the considered optimal control problem is specified below.

In addition to the positivity constraints~\eqref{eq:U(x)}, we impose constraint sets $\mathbb{X} \subseteq \R_{\geq 0}^n$ and $\mathbb{U} \subseteq \R_{\geq 0}^m$. 
We assume that both sets contain the origin.
The states $x \in \mathbb{X}$ are called \emph{admissible}, and for each $x \in \mathbb{X}$, the control values 
\begin{equation*}
    u \in \mathbb{U}(x) \coloneqq U(x) \cap \mathbb{U}    
\end{equation*}
are called {admissible control values (for $x$)}. 
For ${N \in \N}$, a {finite control of length $N$} is a map ${u \colon \{0,\dots,N\!-\!1\} \to \R_{\geq 0}^m}$.
The set of all finite controls of length $N$ is denoted by $U^N\!$. 
Similarly, an {infinite control} is a map ${u \colon \N_0 \to \R_{\geq 0}^m}$, and the set of all infinite controls is denoted $U^\infty\!\!$. 
Given an initial state ${x_0 \in \R_{\geq 0}^n}$ and a control ${u \in U^N}$ or ${u \in U^\infty}\!\!$, the {trajectory} $x_u(\cdot,x_0)$ is created iteratively via $x_u(0,x_0) \coloneqq x_0$ and
\begin{equation*}
	\forall t \in I \colon x_u(t+1,x_0) \coloneqq x_u(t,x_0) + B u(t) ,
\end{equation*}
with $I = \N_0$ if $u \in U^\infty$ and $I = \{0,\dots,N-1\}$ otherwise. 

A {feedback law} is a map ${\mu \colon \R_{\geq 0}^n \to \R_{\geq 0}^m}$.
For an initial state ${x_0 \in \R_{\geq 0}^n}$, the {closed-loop trajectory} $x_\mu(\cdot,x_0)$ is defined iteratively via $x_\mu(0,x_0) = x_0$ and
\begin{equation*}
	\forall t \in \N_0 \colon x_\mu(t+1,x_0) = x_\mu(t,x_0) + B \mu(x_\mu(t,x_0)) .
\end{equation*}

Given a prediction horizon $N \in \N$ and an initial state $x_0 \in \mathbb{X}$, we call a control $u \in U^N$ and the corresponding trajectory $x_u(\cdot,x_0)$ {admissible for $x_0$ up to time $N$} if for each $t \in \{ 0, \dots, N-1 \}$, we have
\begin{equation*}
    x_u(t,x_0) \in \mathbb{X} \land u(t) \in \mathbb{U}(x_u(t,x_0)) \land x_u(N,x_0) \in \mathbb{X}.
\end{equation*}
The set of all admissible controls of length $N$ is denoted by $\mathbb{U}^N(x_0)$. 
Due to the structure of~\eqref{eq:U(x)}%the linear positive system
, for each $x \in \mathbb{X}$, the admissible set $\mathbb{U}(x)$ contains the zero vector, which in particular implies for each $N \in \mathbb{N}$ that $\mathbb{U}^N(x) \neq \emptyset$. 
An infinite control ${u \in U^\infty}$ is called {admissible for ${x_0 \in \mathbb{X}}$} if for each ${N \in \N}$, its restriction ${\left. u \right|_{[0,N)}}$ is admissible for $x_0$ up to time $N.$ 
The set of admissible controls for $x_0$ is denoted by $\mathbb{U}^\infty(x_0)$. 
Finally, a feedback law $\mu$ is called {admissible} if $x \in \mathbb{X}$ implies $\mu(x) \in \mathbb{U}^1(x)$.

We introduce a \emph{linear} stage cost ${\ell \colon \R_{\geq 0}^n \times \R_{\geq 0}^m \to \R_{\geq 0}}$, which takes the form
\begin{equation*}
    \ell(x,u) \coloneqq s^\top x + r^\top u ,
\end{equation*}
with given weights $s \in \R_{>0}^n$ and $r \in \R_{\geq0}^m$. 
The cost over a horizon $N \in \N$ for an initial state $x_0 \in \mathbb{X}$ and an admissible control $u \in \mathbb{U}^N(x_0)$ of length $N$ is then defined as
\begin{equation*}
    J_N(x_0,u) \coloneqq \sum_{t=0}^{N-1} \ell(x_u(t,x_0),u(t)) .
\end{equation*}
Based on this, we introduce the \emph{optimal value function} $V_N$ that minimises for a given ${x_0 \in \mathbb{X}}$ the cost $J_N(x_0,\cdot)$ over the \emph{admissible} controls of length $N$, \ie we introduce
\begin{equation} \label{eq:value_function_finite}
	V_N(x_0) \coloneqq \inf \{ J_N(x_0,u) \mid u \in \mathbb{U}^N(x_0) \} .
\end{equation}
If $u \in \mathbb{U}^N(x_0)$ satisfies $V_N(x_0) = J_N(x_0,u)$, we say that~$u$ is {an optimal control of length $N$ for $x_0$}.
For simplicity of exposition, we assume that for each state $x_0 \in \mathbb{X}$ there exists an optimal control of length $N$ for $x_0$.

Similarly, we define the {infinite-horizon cost} for $x_0 \in \mathbb{X}$ and an admissible control $u \in \mathbb{U}^\infty(x_0)$ as
\begin{align*}
	J_\infty(x_0,u) \coloneqq \lim_{N \to \infty} J_N(x_0, \left. u \right|_{[0,N)} ) .
\end{align*} 
The corresponding optimal value function is
\begin{equation} \label{eq:valuze_function_inf}
	V_\infty(x_0) \coloneqq \inf \{ J_\infty(x_0,u) \mid u \in \mathbb{U}^\infty(x) \} .
\end{equation}
If the optimal value $V_\infty(x_0)$ is finite, a control $u \in \mathbb{U}^\infty(x_0)$ that satisfies ${J_\infty(x_0,u) = V_\infty(x_0)}$ is called an {optimal control for $x_0$}. 
The following standard observation, the proof of which can be found in, \eg \cite[Chapter~6]{Gruene2017}, is useful later. % in this work.
\begin{lemma} \label{lemma:increasing_value_function}
	Let $N \in \mathbb{N}$. 
	For each admissible state $x \in \mathbb{X}$, we have $V_{N}(x) \leq V_\infty(x)$. 
\end{lemma}
Finally, given an admissible feedback law $\mu \colon \mathbb{X} \to \mathbb{U}$, the {closed-loop cost} for $x_0 \in \mathbb{X}$ associated with $\mu$ is
\begin{equation*}
    J_\infty^\mathrm{cl}(x_0,\mu) \coloneqq \lim_{K \to \infty} \sum_{t=0}^{K-1} \ell(x_{\mu}(t,x_0), \mu(x_{\mu}(t,x_0))) .
\end{equation*}

\section{Solving the Optimal Control Problem}\label{sec:solving}

\noindent
The considered class of optimal control problems is of particular interest because -- for a specific case -- they can be solved analytically.
Moreover, the corresponding optimal solution has a clear interpretation with respect to the underlying graph structure. 
This case arises if we consider positivity constraints only, \ie ${\mathbb{X} = \R_{\geq 0}^n}$ and ${\mathbb{U} = \R_{\geq 0}^m}$. 
Note that ${\mathbb{U} = \R_{\geq 0}^m}$ implies ${\mathbb{U}(x) = U(x)}$.
In this section, we first recall results from the previous work \cite{Ohlin2024}. 
We then discuss why additional capacity constraints $\mathbb{X}$ and $\mathbb{U}$ significantly complicate the problem, and present model predictive control as a remedy. 
%We recall an explicit solution to the infinite-horizon optimal control problem.
Theorem~\ref{theo:positive_infinite_horizon} gives an explicit solution to the infinite-horizon problem. % is available.
\begin{theorem}[{\cite[Theorem~1]{Ohlin2024}}] \label{theo:positive_infinite_horizon}
    Let ${\mathbb{X} = \R_{\geq 0}^n}$ and ${\mathbb{U} = \R_{\geq 0}^m}$.
	Then, the optimal value function $V_\infty$ attains a finite value for each $x \in \R_{\geq 0}^n$ if and only if there exists $p \in \R_{>0}^n$ satisfying
	\begin{equation*}
		s +  \sum_{i=1}^n \min \{ r_i + B_i^\top p, 0 \} e_i = 0 .
	\end{equation*}
	In this case, the value for ${x \in \R_{\geq 0}^n}$ is ${V_\infty(x) = p^\top x}$.
\end{theorem}
\vspace{-1em}
\pagebreak
Throughout this subsection, we consider the case ${\mathbb{X} = \R_{\geq 0}^n}$ and ${\mathbb{U} = \R_{\geq 0}^m}$, and assume that a ${p \in \R_{>0}^n}$ as in Theorem~\ref{theo:positive_infinite_horizon} exists.
For the considered system class, this assumption is equivalent to the existence of a path from each vertex $i \in \mathcal{V}$ to the goal vertex $g$. 
Then, we can show that an optimal control exists for each $x \in \R_{\geq 0}^n$.

In particular, the optimal control takes feedback form, \ie there exists an admissible feedback law $\mu \colon \R_{\geq 0}^n \to \R_{\geq 0}^m$ such that for each $x_0 \in \R_{\geq 0}^n$, the control $u \in \mathbb{U}^\infty(x_0)$ defined by
\begin{equation} \label{eq:optimal_positive_control}
	u(t) = \mu(x_{\mu}(t,x_0)) 
\end{equation}
is optimal. % for $x$. 
To define $\mu$, construct the \emph{local} feedback gain %$K_i$ in node~$i$
\begin{equation} \label{eq:local_feedback}
   {K}_i = \Matrix{\mathbf{0}_{(k-1)\times n}\\e_i^\top\\\mathbf{0}_{(m_i-k)\times n}} 
   % {K}_i^\top = \Matrix{\mathbf{0}_{n \times (k-1)} & e_i & \mathbf{0}_{n \times (m_i-k)}} ,
\end{equation}
where $k$ is the (first) index of the minimal element of ${r_i + B_i^\top p}$. 
Here, the partition of $r$ follows the partition~\eqref{eq:B}. The global gain is then ${{K}=\Matrix{{K}_1^\top \dots {K}_n^\top}^\top}$ and the corresponding feedback law ${\mu(x) = {K}x}$. 
We refer to $\mu$ as the \emph{optimal positive feedback law}. 
It can be shown that for each $x_0 \in \R_{\geq 0}^n$, the control $u \in \mathbb{U}^\infty(x_0)$ defined by \eqref{eq:optimal_positive_control} is an optimal control for $x_0$ \cite[Theorem~1]{Ohlin2024}.%by Theorem~\ref{theo:positive_infinite_horizon}. 

The graph interpretation of $\mu$ is as follows.
Each vertex ${i \in \mathcal{V}}$ sends \emph{all} of the commodity in its buffer to the successor defined by $B_i{K}_i$, \ie the tail of the edge associated with the $k$-th column of $B_i$. 
We denote vertex~$i$'s successor by ${\nu(i) \in \mathcal{V} \cup \{g\}}$. Hence, $\mu$ selects for each vertex ${i \in \mathcal{V}}$ a unique successor ${j = \nu(i)}$ from those accessible through $B_i$. 
Effectively, $\mu$ removes all but one of the outgoing edges from the graph. 
Hence, in the resulting graph, there exists at most one path between any two vertices.
In addition, Theorem~\ref{theo:positive_infinite_horizon} implies that for each vertex ${i \in \mathcal{V}}$, there exists a path from $i$ to $g$, obtained by repeatedly applying $\nu$. 

Let~$\Tilde{B}$ be the matrix containing only the positive elements of~$B$, representing successor vertices for the inputs. % in $u$.
As a consequence of the structure of $K$, we get
\begin{equation}\label{eq:Btilde}
    BK = \Tilde{B}K - I.
\end{equation}
Note that, since all commodity leaves the system in finite time under~$\mu$,~$\tilde{B}K$ is nilpotent. \setcounter{example}{0}
\begin{example} \textit{(contd.)} \label{ex:optimal_positive_controller}
    Consider state costs ${s^\top\!\! = \!(10 \; 5 \; 1 \; 3 \; 2)}$, and control value costs ${r_1^\top = (1 \; 5 \; 5)}$, ${r_2 = \mathbf{1}}$, ${r_3 = r_4 = r_5 = 1.}$
    Then, ${p^\top\! = (19 \; 8 \; 2 \; 6 \; 5)}$ satisfies the condition in Theorem~\ref{theo:positive_infinite_horizon}. 
    Hence, $\nu$ is given by ${\nu(1) = 2}$, ${\nu(2) = 3}$, ${\nu(3) = 6}$, ${\nu(4) = \nu(5) = 3.}$
    The edges used by the resulting optimal positive feedback law are highlighted in green in Fig.~\ref{fig:shortest_path_problem}. 
    The edges $(1,4)$, $(1,5)$, $(2,4)$ and $(2,5)$ are not used by $\mu$. % in closed-loop.
\end{example}

Next, we treat the case where $\mathbb{X}$ and $\mathbb{U}$ are proper subsets of~$\R_{\geq 0}^n$ and $\R_{\geq 0}^m$, respectively. 
Specifically, we impose upper bounds on the states and control values, described by the sets
\begin{equation} \label{eq:box_constraints}
\begin{aligned}
	\mathbb{X} &\coloneqq \{
		x \in \R_{\geq 0}^n \mid x \leq \bar{x}
	\}, \\
	\mathbb{U} &\coloneqq \{
		(u_1, \dots, u_n) \in \R_{\geq 0}^m \mid \forall i \in \mathcal{V} \colon u_i \leq \bar{u}_i
	\},
\end{aligned}
\end{equation}
with given bounds $\bar{x} \in \R^n_{>0}$ and $\bar{u}_i \in \R_{>0}^{m_i}$. 
\pagebreak

% \noindent {\it Model Predictive Control}
\noindent {\it Model Predictive Control:} 
%We give a brief overview of MPC for readers unfamiliar with the approach.
In an MPC scheme, the optimal control problem is solved only over a {finite} horizon $N \in \N$ in receding horizon fashion \cite{Gruene2017,Borrelli2017}. 
Based on the solution to the finite-horizon problem, an admissible {MPC feedback law} ${\mu_N \colon \mathbb{X} \to \mathbb{U}}$ is constructed. %, which is applied in closed-loop. 
For any initial state ${x_0 \in \mathbb{X}}$, this leads to the closed-loop trajectory ${x_{\mu_N}(\cdot,x_0)}$. Specifically, the MPC feedback law $\mu_N$ is defined as follows. % in our case. 
For any ${x \in \mathbb{X}}$, let ${u^\star \in \mathbb{U}^N(x)}$ be an optimal control of length~$N$ for~$x$.
We then set ${\mu_N(x) \coloneqq u^\star(0)}$. 
Note that this definition does not require uniqueness of $u^\star$. 
For $\mu_N$ to be well-defined, suppose that for each $x \in\mathbb{X}$ we select one specific control from the set of optimal controls of length $N$. % for $x$. 
The feedback law~$\mu_N$ is admissible by construction. Furthermore, recursive feasibility of $\mu_N$ is guaranteed for all $x \in \mathbb{X}$, since $0 \in \mathbb{U}(x)$ implies that the zero control is always admissible, ensuring $\mathbb{U}^N(x) \neq \emptyset$ for all $N \in \mathbb{N}$.

From this construction alone, it is not immediately clear whether $\mu_N$ stabilises the closed-loop system. 
Similarly, it is not obvious how near the closed-loop cost $J_\infty^\mathrm{cl}(\cdot, \mu_N)$ of the MPC feedback law is to the optimal value function $V_\infty$.
The key to guaranteeing stability %for this scheme 
is boundedness of~$V_N$.

\begin{assumption} \label{ass:bound_optimal_value_function}
    There exists ${\gamma > 0}$ such that for each ${N \in \N}$
    \begin{equation*}
        \forall x \in \mathbb{X} \colon V_N(x) \leq \gamma s^\top x .
    \end{equation*}
\end{assumption}
Under this assumption, which is validated below in Section~\ref{sec:bounding}, we can derive a lower bound on the prediction horizon to guarantee stability of the closed loop. 
Moreover, a suboptimality estimate directly in terms of $V_\infty$ is available.
\begin{theorem} \label{theo:mpc_closed_loop}
    Let Assumption~\ref{ass:bound_optimal_value_function} hold, and suppose that the horizon $N \in \mathbb{N}$ satisfies
	\begin{equation*}
		N > 2 + \frac{\ln(\gamma-1)}{\ln \gamma - \ln(\gamma-1)} .
	\end{equation*}
	Then, the origin of the MPC closed-loop system is globally asymptotically stable on $\mathbb{X}$. 
	In addition, with
	\begin{equation*}
		(0,1) \ni \alpha_N = 1 - \frac{(\gamma - 1)^{N}}{\gamma^{N-1} - (\gamma - 1)^{N-1}} ,
	\end{equation*}
	we have for each $x_0 \in \mathbb{X}$ the suboptimality estimate
	\begin{equation*}
		J_\infty^\mathrm{cl}(x_0,\mu_N) \leq V_\infty(x_0) / \alpha_N .
	\end{equation*}
\end{theorem}
\begin{proof}
    We seek to apply \cite[Theorem~6.20]{Gruene2017}. 
    Let ${x \in \mathbb{X}}$ be arbitrary.
    By non-negativity of $r$ and since ${0 \in \mathbb{U}(x)}$ holds, we have ${\ell^\star(x) \coloneqq \min \{ \ell(x,u) \mid u \in \mathbb{U}(x) \} = s^\top x}$. 
    In addition, since the entries of $s$ are positive, there exist ${\alpha_1, \alpha_2 \in \mathcal{K}_\infty}$ such that ${\alpha_1(\| x \|_1) \leq \ell^\star(x) \leq \alpha_2(\| x \|_1)}$.
    Therefore, the assertions follow from \cite[Theorem~6.20]{Gruene2017}, by applying \cite[Proposition~6.18]{Gruene2017} with $B_K \colon r\mapsto \gamma r$.
\end{proof}

\begin{remark}
    In linear MPC, the typical approach to guarantee closed-loop stability is to solve a modified version of the optimal control problem by adding {terminal conditions}. 
    In these schemes, a suitably chosen {terminal set $\mathbb{X}_\mathrm{f} \subseteq \mathbb{X}$} is introduced, and the finite-horizon cost is changed to
    \begin{equation*}
        \sum_{t=0}^{N-1} \ell(x_u(t,x_0),u(t)) + F(x_u(N,x_0)) ,
    \end{equation*}
    where ${F \colon \mathbb{X}_\mathrm{f} \to \R_{\geq 0}}$ is a {terminal cost}. 
    Moreover, in addition to requiring ${u \in \mathbb{U}^N(x_0)}$, these schemes add the terminal constraint ${x_u(N,x_0) \in \mathbb{X}_\mathrm{f}}$. 
    {The arguably most intuitive solution to the problem at hand, namely the choice $F \equiv 0$ and $\mathbb{X}_\mathrm{f} = \{0\}$, falls in this framework, but is undesirable in practice for several reasons; see, \eg the discussion in \cite[Section~12.6]{Borrelli2017}.}
    Such terminal conditions %, \eg in linear-quadratic optimal control, 
    are {more} often based on the unconstrained linear-quadratic regulator. 
    
    The approach using terminal conditions is indeed also a viable strategy in the present setting, relying on the optimal positive feedback law $\mu$. We do not follow this strategy for two reasons. 
    First, adding the additional terminal constraint $\mathbb{X}_\mathrm{f}$ to the optimal control problem generally reduces the region of attraction of the closed MPC loop. 
    Typically, a large horizon $N$ is needed for a large region of attraction. 
    Second, the addition of the terminal cost~$F$ makes it significantly harder to estimate the difference between $J_\infty^\mathrm{cl}(\cdot,\mu_N)$ and the %original 
    optimal value function~$V_\infty$. 
    To obtain a good approximation in terms of $V_\infty$ in the case with terminal conditions, we would need ${N \to \infty}$, which we would like to avoid.
    For the outlined reasons, we choose an MPC scheme without terminal conditions. 
    We refer to \cite[Chapter~5 and Section~7.4]{Gruene2017} for a more detailed discussion. 
\end{remark}

\section{Bounding the Optimal Value Functions}\label{sec:bounding}

\noindent
The goal of this section is to verify Assumption~\ref{ass:bound_optimal_value_function}. 
Our approach is to design a suboptimal but linear admissible feedback law ${\hat{\mu} \colon \mathbb{X} \to \mathbb{U}}$. 
The structure of $\hat{\mu}$ admits an explicit expression for $J_\infty^\mathrm{cl}(\cdot,\hat{\mu})$.
This, together with Lemma~\ref{lemma:increasing_value_function}, yields a bound on the finite-horizon optimal value functions{, enabling stability and performance guarantees of the MPC closed-loop system by leveraging Theorem~\ref{theo:mpc_closed_loop}.}

We focus on a suboptimal feedback law $\hat{\mu}$ that has the same structure as the optimal positive feedback law ${\mu \colon \R_{\geq 0}^n \to \R_{\geq 0}^m}$. With additional constraints as in~\eqref{eq:box_constraints}, $\mu$ is not in general admissible. From a graph perspective, the following two scenarios result in an inadmissible policy:

$(i)$ The optimal positive feedback law identifies for each vertex ${i \in \mathcal{V}}$ a successor vertex ${j = \nu(i)}$ and sends \emph{all} of the commodity stored in $i$ to $j$. 
This could violate the control value constraint on the edge from $i$ to $j$.

$(ii)$ The full contents of a vertex $i$ are routed away at every time step under $\mu$. Despite this, however, the total incoming commodity from all vertices immediately upstream of $i$ may violate the state constraint. This requires some coordination among upstream vertices.
\setcounter{example}{0}
\begin{example} \textit{(contd.)}
    % We continue Example~\ref{ex:optimal_positive_controller}. 
    Given are the state bounds $\bar{x} = \mathbf{1}$ and the control value bounds $
        \bar{u}_1 = \bar{u}_2 = (1/4) \mathbf{1}, \quad 
        % \bar{u}_2 &= (1 \; 2 \; 2)^\top, &
        \bar{u}_3 = \bar{u}_4 = \bar{u}_5 = 1. $
 Here, $\mu$ is \emph{not} an admissible feedback law.
    For instance, consider the state ${x = (0 \; 1 \; 0 \; 0 \; 0)^\top \in \mathbb{X} .}$ 
    Then, $\mu(x) \notin \mathbb{U}(x)$, since the control value constraint of edge $(2,3)$ is violated. 
    That is an example of scenario $(i)$. 
    An example of scenario $(ii)$ is given by ${x = (0 \; 1/4 \; 0 \; 1 \; 0)^\top \in \mathbb{X}.}$
    Then, ${\mu(x) \in \mathbb{U}(x)}$ but ${\mu(x) \notin \mathbb{U}^1(x)}$, since ${x_\mu(1,x) = (0 \; 0 \; 5/4 \; 0 \; 0) \notin \mathbb{X}}$.
\end{example}
        
Thus, the main problem of the optimal positive feedback law regarding constraint satisfaction is the fact that \emph{all} of the commodity stored in each vertex is transferred at every time step. 
Motivated by this discussion, we design the admissible feedback law $\hat{\mu}$ by keeping for each vertex $i \in V$ the successor vertex $\nu(i)$ identified by $\mu$, but reducing the commodity sent over the identified edges at every time step by some factor. 
To make this precise, assume that ${\lambda \in (0,1]^n}$ is given and set $\Lambda = \operatorname{diag}(\lambda)$.
Define $\hat{\mu}$ by
\begin{equation}\label{eq:muhat}
    \hat{\mu}(x) \coloneqq K \Lambda x ,
\end{equation}
which first scales down each element $x_i$ by~$\lambda_i$, then follows the path identified by $\mu$.
As we establish next, it is possible to render $\hat{\mu}$ admissible by an appropriate choice of~$\lambda$.
\begin{proposition}\label{prop:admissible}
    Consider the feedback law $\hat{\mu}$ in \eqref{eq:muhat} and let
    \begin{equation}\label{eq:L}
        L = \{\lambda \in (0,1]^n \mid K\Lambda \bar{x} \leq \bar{u} \textnormal{ and } B K \Lambda \bar{x} \leq 0\}\; {\neq \emptyset} .
    \end{equation}
    {Then,} $\hat{\mu}$ is admissible if and only if ${\lambda\in L}$.
\end{proposition}
\begin{proof}
    We first show that any $\lambda \in (0,1]^n$ preserves positivity without additional assumptions.
    Let ${x \in \mathbb{X}}$ be arbitrary. 
    We have ${K \Lambda x \leq K x}$, so that ${\hat{\mu}(x) \in U(x)}$ follows immediately from the definition~\eqref{eq:U(x)}. 
    {Recall that the inequality ${x\le\bar{x}}$ is preserved when multiplied by a nonnegative matrix. Since $K\Lambda\ge0$}, %for each $x \in \mathbb{X}$ 
    we have for the control value constraint
    \begin{align*}
        {\forall x \in X \colon} \hat{\mu}(x) \in\mathbb{U}  %\; \textnormal{for all} \; x\in\mathbb{X}\\
        &\iff {\forall x \in X \colon} K\Lambda x\le\bar{u} \\  %\; \textnormal{for all} \; x\in\mathbb{X}\\
        &\iff  K\Lambda\bar{x} \le \bar{u}.
    \end{align*}
    Moreover, for each ${x \in \mathbb{X}}$, the inclusion ${\hat{\mu}(x)\in U(x)}$ additionally implies $(I+BK\Lambda)x \leq (I+BK\Lambda)\bar{x}$, and thus 
    \begin{align*}
         &{\forall x \in X \colon} x_{\hat{\mu}}(1,x) \in \mathbb{X} &
         &\!\!\!\!\!\!\!\iff 
         {\forall x \in X \colon} (I+BK\Lambda)x \le \bar{x} \\ 
         &\iff  
         (I+BK\Lambda)\bar{x}\le\bar{x} &
         &\!\!\!\!\!\!\!\iff 
         BK\Lambda\bar{x}\le0 .
    \end{align*}
    Combined, this yields the sought equivalence. 
    
    To show that ${L\neq \emptyset}$, first note that $K\Lambda\bar{x}\le\bar{u}$ can be satisfied by making~$\lambda$ sufficiently small. %, since all factors on the left-hand side are non-negative. 
    Likewise, if $\lambda$ satisfies $BK\Lambda\bar{x}\le0$, then so does any $\alpha\lambda$ with $\alpha\in(0,1]$. 
    We show the existence of such a $\lambda$ by applying~\eqref{eq:Btilde} to get
    \begin{equation}\label{eq:BtildeLambda}
        BK\Lambda\bar{x} \le 0 \iff \tilde{B}K\Lambda\bar{x} \le \Lambda\bar{x}.
    \end{equation}
    Since $\tilde{B}K$ is nilpotent, there exists ${v \in \R_{>0}^n}$ such that ${\tilde{B}Kv\le v}$~\cite[Theorem~8.3.1]{Horn2013}. We choose $\lambda$ to satisfy $\Lambda\bar{x} = \alpha v$ for some ${\alpha\in(0,1]}$, which can be taken sufficiently small to satisfy both inequalities. 
\end{proof}

So far, we have established that a suitable choice of~$\lambda$ ensures admissibility of $\hat{\mu}$. 
Next, we derive an explicit linear expression for the closed-loop infinite horizon cost $J_\infty^\mathrm{cl}(\cdot,\hat{\mu})$. 

\begin{lemma} \label{lemma:admissible_lyapunov}
    Assume that $p$ according to Theorem~\ref{theo:positive_infinite_horizon} exists and let ${\lambda \in L}$. 
    Then,
    $J^{\mathrm{cl}}_{\infty}(x_0,\hat{\mu})=\hat{p}^\top x_0$ with $\hat{p}\;{\in\R^n_{>0}}$ given by
    \begin{equation}\label{eq:phat}
        \hat{p} = {-} (BK)^{-\top}(\Lambda^{-1}s+K^\top r).
    \end{equation}
\end{lemma}
\begin{proof}
    Since $\hat{\mu}$ is admissible for $\lambda\in L$, the closed-loop cost can be found by evaluating the Bellman equation with fixed policy $\hat{\mu}$ and ansatz $J^{\mathrm{cl}}_{\infty}(x,\hat{\mu}) = \hat{p}^\top x$, yielding %(see \cite[Proposition 6]{Ohlin2025}). 
    \begin{align*}
        && \!\!\!\!\!\!\!\!\!\!\!\!\!\!\!\!\!\!\!\!s^\top\! x + r^\top\! K\Lambda x + \hat{p}^\top \!(I+BK\Lambda)x &=\hat{p}^\top x\\
        &\iff&(I-(I+BK\Lambda)^\top)\hat{p} &= s+\Lambda^\top K^\top r\\
        &\iff& {-} (BK)^\top\hat{p} &= \Lambda^{-1}s+K^\top r.
        % &\iff& \hat{p} &= (\Lambda^{-1} s +K^\top r
    \end{align*}
    From here, we arrive at~\eqref{eq:phat} by noting that~\eqref{eq:Btilde} implies invertibility {of~$-(BK)^\top$ since $\tilde{B}K$ is nilpotent. 
    Moreover, $-(BK)^\top$ is an M-matrix by construction and therefore has a nonnegative inverse. 
    Since $-(BK)^{-\top}$ has full rank, it cannot contain a zero row, so $\Lambda^{-1} s + K^\top r > 0$ implies ${\hat{p} > 0}$.  
    }
\end{proof}

{Lemma~\ref{lemma:admissible_lyapunov} allows us} to verify Assumption~\ref{ass:bound_optimal_value_function}.
\begin{proposition} \label{prop:value_function_bound}
    Let the assumptions of Lemma~\ref{lemma:admissible_lyapunov} hold. 
    Then, Assumption~\ref{ass:bound_optimal_value_function} is satisfied with $\gamma = \max \{ \hat{p}_i / s_i \mid i \in \mathcal{V} \}$. 
\end{proposition}
\begin{proof}
	By definition of $V_\infty$ %and Corollary~\ref{coro:suboptimal_closed_loop_cost}, 
    we have for each $x \in \mathbb{X}$ that
	\begin{equation*}
		V_\infty(x) \leq J_\infty^\mathrm{cl}(x,\hat{\mu}) = \hat{p}^\top x.
	\end{equation*}
    By the stated definition, $\gamma$ is the smallest number such that ${\hat{p}^\top\! x\le \gamma s^\top\! x}$. 
    By Lemma~\ref{lemma:increasing_value_function}, ${V_N(x) \le \gamma s^\top\! x}$ as desired.
\end{proof}

Finally, we discuss how to select {an optimal} $\lambda \in L$ based on the findings in Theorem~\ref{theo:mpc_closed_loop}. 
From there, it is immediate that a small bound $\gamma$ is desirable. 
By Proposition~\ref{prop:value_function_bound}, we are thus interested in computing %the \emph{optimal bound}
\begin{equation} \label{eq:optimal_gamma}
    \gamma^\star \coloneqq \min \{ \max \{ \hat{p}_i / s_i \mid i \in \mathcal{V} \} \mid \lambda \in L \} ,
\end{equation}
where $\hat{p}$ is defined by \eqref{eq:phat}.
The constraint ${\lambda > 0}$ makes it natural to formulate the design problem as a \emph{geometric program} \cite[Section~4.5]{Boyd2004}.

\begin{proposition} \label{prop:geometric_program}
    The optimal bound $\gamma^\star$ in \eqref{eq:optimal_gamma} can be obtained by solving {the geometric program} %\eqref{eq:geometric_program}
    {\begin{mini}
        {(\gamma,\lambda)}{\gamma \;\;\;\;\;\;\;\;\;\;\;\;\;\;\;\;\;\;\;\;\;\;\;\;\;\;\;\;\;\;\;\;\;\;\;\;\;\;\;\;\;\;\;\;\;\;\;\;\;\;\;\;\;\;\;}
        {\label{eq:geometric_program}}{}
        \addConstraint{\lambda \in}{\;L}
        \addConstraint{\gamma s \ge }{-(BK)^{-\top}(\Lambda^{-1} s + K^\top r).}
    \end{mini}}
\end{proposition}

\begin{proof}
    {The second constraint ensures that, for each $i\in\mathcal{V}$, we have $\gamma \ge \hat{p}_i/s_i$, with $\hat{p}$ given by~\eqref{eq:phat}. 
    Hence, the value of~\eqref{eq:geometric_program} is $\gamma^*$ as in~\eqref{eq:optimal_gamma}.
    It is left to show that \eqref{eq:geometric_program} is a geometric program.}
    {This requires all inequality constraints to be element-wise posynomials and the decision variables to be strictly positive.}
    
    {Consider the constraints of~\eqref{eq:L}. 
    The positivity constraint $\lambda \in (0,1]^n$ is immediately in posynomial form.
    %The first is immediately on posynomial form and yields nonnegativity of $\lambda$. 
    The capacity constraint $K \Lambda \bar{x} \leq \bar{u}$ is posynomial in each row, since ${K\ge0}$ and ${\bar{x}, \bar{u} >0}$. 
    We can rewrite the upstream state constraint $B K \Lambda \bar{x} \leq 0$ using~\eqref{eq:Btilde} as
    \begin{equation}\label{eq:reform}
        B K \Lambda \bar{x} \leq 0 \iff \Lambda^{-1}\Tilde{B}K\Lambda\bar{x}\le \bar{x}.
    \end{equation}
    Recall that ${\Lambda = \textnormal{diag}(\lambda)}$, so $\Lambda^{-1}$ has the element-wise inverses of~$\lambda$ along the diagonal. 
    Moreover, ${\Tilde{B}, K \ge 0}$ and ${\bar{x} > 0}$, so the rows of $\Tilde{B} K \Lambda \bar{x}$ are posynomials in $\lambda$. 
    The division of a posynomial by a monomial yields a posynomial, which implies that the rows of $\Lambda^{-1} \Tilde{B} K \Lambda \bar{x}$ are posynomials in $\lambda$.
    In the final constraint of~\eqref{eq:geometric_program}, Lemma~\ref{lemma:admissible_lyapunov} yields that the right-hand side is positive, and thus ${s > 0}$ implies ${\gamma > 0}$. 
    The matrix ${-(BK)^{-\top}}$ is nonnegative as argued in the proof of Lemma~\ref{lemma:admissible_lyapunov}, and $K^\top r$ is nonnegative by construction.
    Dividing by $\gamma$, all coefficients in the right-hand side product are nonnegative, yielding a posynomial in $(\gamma,\lambda)$.
    
    Finally, since $\bar{x}, \bar{u}$ and $s$ are all element-wise positive, the upper bounds of all constraints in \eqref{eq:geometric_program} can be normalised to~$\mathbf{1}$. 
    This finishes the proof.}
\end{proof}

%%%%%%%%%%%%%%%%%%%%%%%%%%%%%%%%%%%%%%%%%%%%%%%%%%%%%%%%%%%%%%%%%%%%%%%%%%%%%%%%
\section{Numerical Example}\label{sec:examples}
\noindent
We demonstrate the MPC solution for the system defined in Example~\ref{ex:incidence_matrix}. 
First, we solve the geometric program in Proposition~\ref{prop:geometric_program}, which gives the optimal bound ${\gamma^\star = 6.4}$ for ${\lambda^\top\!= (.25 \; .25 \; 1 \; .29 \; .31)}$. 
Consequently, the minimal stabilising horizon according to Theorem~\ref{theo:mpc_closed_loop} is $N_0 = 12$. 
We select $N$ to be the smallest horizon that yields $\alpha_N > 0.5$ for the suboptimality estimate, which is $N = 16$ with ${\alpha_N = 0.54}$. 
We simulate the closed-loop system under the MPC feedback law $\mu_N$ starting from the initial state $\bar{x}$, and compare it with the closed-loop system under $\hat{\mu}$. 
The resulting trajectories are shown in Fig.~\ref{fig:simulation_states} and Fig.~\ref{fig:simualtion_controls}. 
We can see that MPC steers the system to the origin in a finite number of steps, while $\hat{\mu}$ displays the expected asymptotic behaviour.
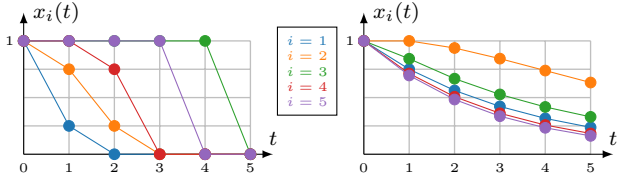
\begin{figure}
\centering
\begin{tikzpicture}[>=latex,x=0.6cm,y=1.5cm]
% MPC --states
\begin{scope}[xshift=-2.25cm]
    % grid
    \draw[lightgray, line width=0.5pt, ystep=0.25, xstep=1] (0,0) grid (5,1);
    \draw[->] (0,0) -- node[pos=1, above] {\footnotesize $t$} (5.5,0);
    \draw[->] (0,0) -- node[pos=1, right] {\footnotesize $x_i(t)$} (0,1.25);
    % labels
    \foreach \t in {0, ..., 5} {
        \node[below] at (\t,0) {\tiny $\t$};
    }
    \foreach \t in {1} {
        \node[left] at (0,\t) {\tiny $\t$};
    }

    % vertex #1
    \draw[myblue] plot[mark=*] coordinates {
        (0,1) (1,0.25) (2,0) (3,0) (4,0) (5,0)
    };
    
    % vertex #2
    \draw[myorange] plot[mark=*] coordinates {
        (0,1) (1,0.75) (2,0.25) (3,0) (4,0) (5,0)
    };
    % vertex #3
    \draw[mygreen] plot[mark=*] coordinates {
        (0,1) (1,1) (2,1) (3,1) (4,1) (5,0)
    };
    % vertex #4
    \draw[myred] plot[mark=*] coordinates {
        (0,1) (1,1) (2,0.75) (3,0) (4,0) (5,0)
    };
    % vertex #5
    \draw[mypurple] plot[mark=*] coordinates {
        (0,1) (1,1) (2,1) (3,1) (4,0) (5,0)
    };
\end{scope}
% Admisible Feedback Law -- states
\begin{scope}[xshift=2.25cm]
    % grid
    \draw[lightgray, line width=0.5pt, ystep=0.25, xstep=1] (0,0) grid (5,1);
    \draw[->] (0,0) -- node[pos=1, above] {\footnotesize $t$} (5.5,0);
    \draw[->] (0,0) -- node[pos=1, right] {\footnotesize $x_i(t)$} (0,1.25);
    % labels
    \foreach \t in {0, ..., 5} {
        \node[below] at (\t,0) {\tiny $\t$};
    }
    \foreach \t in { 1} {
        \node[left] at (0,\t) {\tiny $\t$};
    }

    % vertex #1
    \draw[myblue] plot[mark=*] coordinates {
        (0,1) (1,0.75) (2,0.5625) (3,0.4219) (4,0.3164) (5,0.2373)
    };
    % vertex #2
    \draw[myorange] plot[mark=*] coordinates {
        (0,1) (1,1) (2,0.9375) (3,0.8437) (4,0.7383) (5,0.6328)
    };
    % vertex #3
    \draw[mygreen] plot[mark=*] coordinates {
        (0,1) (1,0.8427) (2,0.6669) (3,0.5277) (4,0.4173) (5,0.3298)
    };
    % vertex #4
    \draw[myred] plot[mark=*] coordinates {
        (0,1) (1,0.7124) (2,0.5075) (3,0.3615) (4,0.2575) (5,0.1834)
    };
    % vertex #5
    \draw[mypurple] plot[mark=*] coordinates {
        (0,1) (1,0.695) (2,0.483) (3,0.3357) (4,0.2333) (5,0.1621)
    };
\end{scope}
% legend
\begin{scope}[node distance=0.2cm]
    \node[myblue] (l1) at (2.5,1.0) {\tiny $i=1$};
    \node[myorange,below of=l1] (l2) {\tiny $i=2$};
    \node[mygreen,below of=l2] (l3) {\tiny $i=3$};
    \node[myred,below of=l3] (l4) {\tiny $i=4$};
    \node[mypurple,below of=l4] (l5) {\tiny $i=5$};

    \draw (l5.south west) rectangle (l1.north east);
\end{scope}
\end{tikzpicture}
\caption{Evolution of the states $x(\cdot) \coloneqq x_{\bar{\mu}}(\cdot,\bar{x})$ %over $t={0,...,5}$ 
in the example of Section~\ref{sec:examples} using MPC ($\bar{\mu} = \mu_N$, left) and the suboptimal feedback law ($\bar{\mu} = \hat{\mu}$, right).}
\vspace*{-5mm}
\label{fig:simulation_states}
\end{figure} 
The closed-loop costs are given by $J_\infty^\mathrm{cl}(\bar{x},\mu_N) = 56.5$ and $J_\infty^\mathrm{cl}(\bar{x},\hat{\mu}) = 111.97$, respectively.
The main difficulty for $\hat{\mu}$ is the \qq{bottleneck} given by vertex~$3$: the initial commodity stored in vertex~$2$ must pass through it, but vertex~$3$ can only accept commodity up to its capacity constraint $\bar{x}_3 = 1$. 
Due to the fixed structure of~$\hat{\mu}$, commodity is not diverted to the edges $(2,4)$ and $(2,5)$. 
The MPC feedback law, on the other hand, uses these two edges during the first two time steps.

We remark that the minimal stabilising horizon in Theorem~\ref{theo:mpc_closed_loop} crucially depends on the stage cost $\ell$, which is well-known; \cf the discussion \cite[Section~6.6]{Gruene2017}. 
In this example, choosing a high state cost $s_1$ gives a small $N_0$, but $N_0$ can be significantly larger for different choices. 
Systematically designing a stage cost that leads to a short minimal horizon is an interesting future research direction. 

\begin{figure}
\centering
\begin{tikzpicture}[>=latex,x=0.6cm]
% MPC -- controls
\begin{scope}[xshift=-2.2cm]
    % vertex #1
    \begin{scope}[y=4cm,yshift=1.75cm]
        % grid
        \draw[lightgray, line width=0.5pt, ystep=0.125, xstep=1] (0,0) grid (4,0.25);
        \draw[->] (0,0) -- node[pos=1, above] {\footnotesize $t$} (4.5,0);
        \draw[->] (0,0) -- node[pos=1, right] {\footnotesize $u_{1,k}(t)$} (0,0.325);
        % labels
        \foreach \t in {0, ..., 4} {
            \node[below] at (\t,0) {\tiny $\t$};
        }
        \foreach \t in {.25} {
            \node[left] at (0,\t) {\tiny $\t$};
        }

        % edge (1 -> 2)
        \draw[myblue] plot[mark=*] coordinates {
            (0,0.25) (1,0.25) (2,0) (3,0) (4,0)
        };

        % edge (1 -> 4)
        \draw[myorange] plot[mark=*] coordinates {
            (0,0.25) (1,0) (2,0) (3,0) (4,0)
        };

        % edge (1 -> 5)
        \draw[mygreen] plot[mark=*] coordinates {
            (0,0.25) (1,0) (2,0) (3,0) (4,0)
        };
    \end{scope}
    % vertex #2
    \begin{scope}[y=4cm,yshift=0cm]
        % grid
        \draw[lightgray, line width=0.5pt, ystep=0.125, xstep=1] (0,0) grid (4,0.25);
        \draw[->] (0,0) -- node[pos=1, above] {\footnotesize $t$} (4.5,0);
        \draw[->] (0,0) -- node[pos=1, right] {\footnotesize $u_{2,k}(t)$} (0,0.325);
        % labels
        \foreach \t in {0, ..., 4} {
            \node[below] at (\t,0) {\tiny $\t$};
        }
        \foreach \t in {.25} {
            \node[left] at (0,\t) {\tiny $\t$};
        }

        % edge (2 -> 3)
        \draw[myblue] plot[mark=*] coordinates {
            (0,0.25) (1,0.25) (2,0.25) (3,0) (4,0)
        };

        % edge (2 -> 4)
        \draw[myorange] plot[mark=*] coordinates {
            (0,0.1226) (1,0.25) (2,0.0) (3,0) (4,0)
        };

        % edge (2 -> 5)
        \draw[mygreen] plot[mark=*] coordinates {
            (0,0.1274) (1,0.25) (2,0.0) (3,0) (4,0)
        };
    \end{scope}
    % vertices #3, 4 and 5
    \begin{scope}[y=1.04cm,yshift=-1.75cm]
        % grid
        \draw[lightgray, line width=0.5pt, ystep=0.5, xstep=1] (0,0) grid (4,1);
        \draw[->] (0,0) -- node[pos=1, above] {\footnotesize $t$} (4.5,0);
        \draw[->] (0,0) -- node[pos=1, right] {\footnotesize $u_{i}(t)$} (0,1.25);
        % labels
        \foreach \t in {0, ..., 4} {
            \node[below] at (\t,0) {\tiny $\t$};
        }
        \foreach \t in { 1} {
            \node[left] at (0,\t) {\tiny $\t$};
        }

        % edge (3 -> 6)
        \draw[myblue] plot[mark=*] coordinates {
            (0,1) (1,1) (2,1) (3,1) (4,1)
        };

        % edge (4 -> 3)
        \draw[myorange] plot[mark=*] coordinates {
            (0,0.3726) (1,0.5) (2,0.75) (3,0) (4,0)
        };

        % edge (3 -> 3)
        \draw[mygreen] plot[mark=*] coordinates {
            (0,0.3774) (1,0.25) (2,0) (3,1) (4,0)
        };
    \end{scope}
\end{scope}
% Admisisble Feedback Law -- controls
\begin{scope}[xshift=2.2cm]
    % vertex #1
    \begin{scope}[y=4cm,yshift=1.75cm]
        % grid
        \draw[lightgray, line width=0.5pt, ystep=0.125, xstep=1] (0,0) grid (4,0.25);
        \draw[->] (0,0) -- node[pos=1, above] {\footnotesize $t$} (4.5,0);
        \draw[->] (0,0) -- node[pos=1, right] {\footnotesize $u_{1,k}(t)$} (0,0.325);
        % labels
        \foreach \t in {0, ..., 4} {
            \node[below] at (\t,0) {\tiny $\t$};
        }
        \foreach \t in {.25} {
            \node[left] at (0,\t) {\tiny $\t$};
        }

        % edge (1 -> 2)
        \draw[myblue] plot[mark=*] coordinates {
            (0,0.25) (1,0.1875) (2,0.1406) (3,0.1055) (4,0.0791)
        };

        % edge (1 -> 4)
        \draw[myorange] plot[mark=*] coordinates {
            (0,0) (1,0) (2,0) (3,0) (4,0)
        };

        % edge (1 -> 5)
        \draw[mygreen] plot[mark=*] coordinates {
            (0,0) (1,0) (2,0) (3,0) (4,0)
        };
    \end{scope}
    % vertex #2
    \begin{scope}[y=4cm,yshift=0cm]
        % grid
        \draw[lightgray, line width=0.5pt, ystep=0.125, xstep=1] (0,0) grid (4,0.25);
        \draw[->] (0,0) -- node[pos=1, above] {\footnotesize $t$} (4.5,0);
        \draw[->] (0,0) -- node[pos=1, right] {\footnotesize $u_{2,k}(t)$} (0,0.325);
        % labels
        \foreach \t in {0, ..., 4} {
            \node[below] at (\t,0) {\tiny $\t$};
        }
        \foreach \t in {.25} {
            \node[left] at (0,\t) {\tiny $\t$};
        }

        % edge (2 -> 3)
        \draw[myblue] plot[mark=*] coordinates {
            (0,0.25) (1,0.25) (2,0.2344) (3,0.2109) (4,0.1846)
        };

        % edge (2 -> 4)
        \draw[myorange] plot[mark=*] coordinates {
            (0,0) (1,0) (2,0) (3,0) (4,0)
        };

        % edge (2 -> 5)
        \draw[mygreen] plot[mark=*] coordinates {
            (0,0) (1,0) (2,0) (3,0) (4,0)
        };
    \end{scope}
    % vertices #3, 4 and 5
    \begin{scope}[y=1.04cm,yshift=-1.75cm]
        % grid
        \draw[lightgray, line width=0.5pt, ystep=0.5, xstep=1] (0,0) grid (4,1);
        \draw[->] (0,0) -- node[pos=1, above] {\footnotesize $t$} (4.5,0);
        \draw[->] (0,0) -- node[pos=1, right] {\footnotesize $u_{i}(t)$} (0,1.25);
        % labels
        \foreach \t in {0, ..., 4} {
            \node[below] at (\t,0) {\tiny $\t$};
        }
        \foreach \t in {1} {
            \node[left] at (0,\t) {\tiny $\t$};
        }

        % edge (3 -> 6)
        \draw[myblue] plot[mark=*] coordinates {
            (0,1) (1,0.8427) (2,0.6669) (3,0.5277) (4,0.4173)
        };

        % edge (4 -> 3)
        \draw[myorange] plot[mark=*] coordinates {
            (0,0.2876) (1,0.2049) (2,0.1460) (3,0.1040) (4,0.0741)
        };

        % edge (3 -> 3)
        \draw[mygreen] plot[mark=*] coordinates {
            (0,0.3050) (1,0.2120) (2,0.1473) (3,0.1024) (4,0.0712)
        };
    \end{scope}
\end{scope}
% legend
% vertex #1
\begin{scope}[yshift=1.75cm,node distance=0.2cm]
    \node[myblue] (l1) at (2,0.8) {\tiny $k=1$};
    \node[myorange,below of=l1] (l2) {\tiny $k=2$};
    \node[mygreen,below of=l2] (l3) {\tiny $k=3$};

    \draw (l3.south west) rectangle (l1.north east);
\end{scope}
% vertex #2
\begin{scope}[yshift=0cm,node distance=0.2cm]
    \node[myblue] (l1) at (2,0.8) {\tiny $k=1$};
    \node[myorange,below of=l1] (l2) {\tiny $k=2$};
    \node[mygreen,below of=l2] (l3) {\tiny $k=3$};

    \draw (l3.south west) rectangle (l1.north east);
\end{scope}
% vertices #3, 4, and 5
\begin{scope}[yshift=-1.75cm,node distance=0.2cm]
    \node[myblue] (l1) at (2,0.8) {\tiny $i=3$};
    \node[myorange,below of=l1] (l2) {\tiny $i=4$};
    \node[mygreen,below of=l2] (l3) {\tiny $i=5$};

    \draw (l3.south west) rectangle (l1.north east);
\end{scope}
\end{tikzpicture}
\caption{Controls $u(\cdot) \coloneqq \bar{\mu}(x_{\bar{\mu}}(\cdot,\bar{x}))$, using MPC ($\bar{\mu} = \mu_N$, left) and the admissible feedback law ($\bar{\mu} = \hat{\mu}$, right) in the example of Section~\ref{sec:examples}. The top plot shows the controls associated with vertex~$1$, the middle plot the controls associated with vertex~$2$, and the bottom plot the controls associated with the remaining vertices. In both cases, $u_{1,2}(\cdot) = u_{1,3}(\cdot)$ and $u_{2,2}(\cdot) = u_{2,3}(\cdot)$.}
\vspace*{-5mm}
\label{fig:simualtion_controls}
\end{figure}
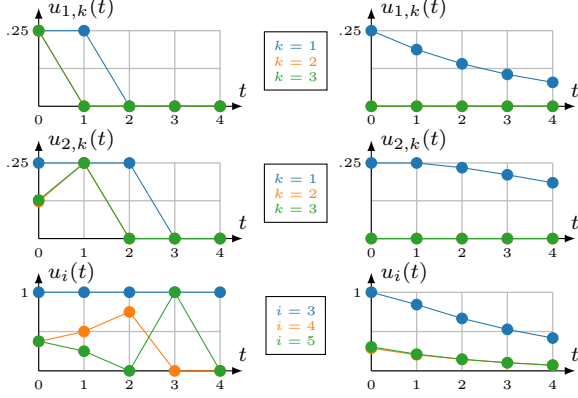

%%%%%%%%%%%%%%%%%%%%%%%%%%%%%%%%%%%%%%%%%%%%%%%%%%%%%%%%%%%%%%%%%%%%%%%%%%%%%%%%
\section{Conclusions}
\noindent
We presented an MPC scheme without terminal conditions for {capacity-}constrained positive systems {in incidence form}, with explicit, graph-computable stability certificates and performance bounds derived from a structured suboptimal controller. The horizon bound is conservative by nature, as the same horizon is prescribed regardless of initial state, while the MPC controller used in practice performs significantly better.
Future work includes improving scalability for large networks, where sacrificing suboptimal controller performance in exchange for a linear programming formulation may be attractive. Extending the framework to more general graph dynamics is also relevant, though the admissibility argument used here does not directly carry over, and potential instability of the dynamics further complicates the synthesis of admissible control laws.
 
\bibliographystyle{IEEEtran}
\bibliography{Ref}

\end{document}